\documentclass[11pt]{article}
\usepackage{amsfonts}
\usepackage{amsmath}
\usepackage{geometry}
\usepackage{amssymb}
\usepackage{graphicx}
\usepackage{mathrsfs}
\usepackage{titlesec}

\newtheorem{theorem}{Theorem}

\newtheorem{lemma}[theorem]{Lemma}

\newtheorem{remark}[theorem]{Remark}

\newenvironment{proof}[1][Proof]{\noindent\textbf{#1.} }{\ \rule{0.5em}{0.5em}}
\titleformat{\section}[block]{\sffamily\Large\bfseries\filcenter}{\thesection}{1em}{}

\begin{document}
\title{A new class of solvable nonlinear difference equation systems}
\author{DURHASAN TURGUT TOLLU\thanks{%
E-mail: dttollu@erbakan.edu.tr 
} \thanks{Conflict of interest: The author has no conflict of interest.} \\ 
Necmettin Erbakan University, Faculty of Science,\\
Department of Mathematics and Computer Sciences, Konya, Turkey}
\maketitle

\begin{abstract}
The paper deals with the following system of nonlinear difference equations 
\begin{equation*}
x_{n+1}=ax_{n}^{2}y_{n}+bx_{n}y_{n}^{2},\
y_{n+1}=cx_{n}^{2}y_{n}+dx_{n}y_{n}^{2},\ n\in \mathbb{N}_{0},
\end{equation*}%
where the initial values $x_{0},y_{0}$ and the parameters $a$, $b$, $c$, $d$
are arbitrary real numbers, which is a new class of solvable systems of
nonlinear difference equations. The general solution of the system is here
obtained in closed form via a practical method.

Keywords: Closed form formula, general solution, system of difference
equations, zero set, trivial solution.

AMS Classification (2010): 39A05, 39A10, 39A20.
\end{abstract}

\section{Introduction}

As in the theory of differential equations, a fundamental problem in the
theory of difference equations is the solvability of these equations or
their systems. If the considered equation or system can be solved, a
detailed analysis of behavior of the solutions can be obtained via its
general solution. Theoretically, the linear difference equations and systems
are always solvable. But, a nonlinear difference equation or a system of
nonlinear difference equations is not always solvable. Because finding the
solution of a nonlinear equation or system requires special methods. These
methods are generally to transform the equations or the system into linear
ones by suitable changes of variables. Accordingly, it is quite interesting
and valuable to find the general solution of a nonlinear difference equation
or system of difference equations.\newline
Although studies on the solvability of difference equations are based on
older times, we see some studies on the topic in the second half of the 20th
century. See, for example, \cite%
{Brand,Grove,Hussein1985,Hussein1986,Hussein1994,HusseinElfiky}. But, for
the last two decades, the topic has become more actual. For some published
recent studies on the solvability problem in difference equations, see, for
example \cite%
{Abo-Zeid,Akrour,Alayachi,Dehghan,El-Dessoky,Gumus,Haddad0,Haddad,KaraTJM,Kurbanli, McGrath,Rhouma,Rhouma1,HS,Stevic2004,Stevic-amc9223,Stevic2017,Taskara2020,T2,T3,Tollu,Yalc,YY6}
and the references in them. One can find some fundamental results in \cite%
{GroveBook,Levy}.\newline
In this paper, we define the following system of nonlinear difference
equations 
\begin{equation}
x_{n+1}=ax_{n}^{2}y_{n}+bx_{n}y_{n}^{2},\
y_{n+1}=cx_{n}^{2}y_{n}+dx_{n}y_{n}^{2},\ n\in \mathbb{N}_{0},  \label{sys}
\end{equation}%
where the initial values $x_{0}$, $y_{0}$ and the parameters $a$, $b$, $c$, $%
d$ are arbitrary real numbers. We obtain the closed form formulas of
nontrivial solutions of system (\ref{sys}) via a practical method. The
method was used for the first time in the reference \cite{Tollu2020} to
solve another system. The method in this study is a slight modification of
the first. See also \cite{TolluYal2021} for another solvable nonlinear
system and another modification.\newline
System (\ref{sys}) belongs to the following class of systems of difference
equations of order one 
\begin{equation}
x_{n+1}=f\left( x_{n},y_{n}\right) ,\ y_{n+1}=g\left( x_{n},y_{n}\right) ,\
n\in \mathbb{N}_{0}.  \label{sysg}
\end{equation}%
\newline
By the solvability of a system in form (\ref{sysg}), we would like to say
that there are sequences of the general terms 
\begin{equation}
\left\{ 
\begin{array}{c}
x_{kn+i}=\phi _{n}^{\left( i\right) }\left( x_{0},y_{0}\right) , \\ 
y_{ln+j}=\psi _{n}^{\left( j\right) }\left( x_{0},y_{0}\right) ,%
\end{array}%
\ 0\leq i\leq k-1,\ 0\leq j\leq l-1,\right.   \label{gsol}
\end{equation}%
that provide the system for $i,j,n\in \mathbb{N}_{0}$ and $k,l\in \mathbb{N}$%
. If there are $n_{0}\in \mathbb{N}_{0}$ such that the identities 
\begin{equation*}
\phi _{n}^{\left( i\right) }\left( x_{0},y_{0}\right) =0\text{ and }\psi
_{n}^{\left( j\right) }\left( x_{0},y_{0}\right) =0
\end{equation*}%
hold for every $n\geq n_{0}$, then we say that the solution represented by (%
\ref{gsol}) is eventually trivial. If $n_{0}=0$, then we say that the
solution is trivial. The set of initial values that produce such solutions
is called zero set of system (\ref{sysg}). See \cite{McGrath} for the zero
set definition of a difference equation.

\section{Some preliminary results}

This section serves us to achieve our main results.\newline
If we take $a=c$, $b=d$ and $x_{0}=y_{0}$ in (\ref{sys}), then we obtain
solutions of the following difference equation 
\begin{equation}
x_{n+1}=\left( a+b\right) x_{n}^{3}  \label{const}
\end{equation}%
from its solutions for all $n\in \mathbb{N}_{0}$. This equation is a special
case of the difference equation 
\begin{equation}
x_{n+1}=a_{n}x_{n}^{3},  \label{gpe}
\end{equation}%
where $\left( a_{n}\right) $ is a sequence with real-terms. If $a_{n}\neq 0$
for all $n\in \mathbb{N}_{0}$, then Equation (\ref{gpe}) has the solution 
\begin{equation*}
x_{n}=x_{0}^{3^{n}}\prod\limits_{k=0}^{n-1}a_{k}^{3^{n-k-1}}.
\end{equation*}%
Several equations in the form of (\ref{gpe}) naturally comes up in our
study. If we allow to become $a_{n}=a+b$ for every $n\in \mathbb{N}_{0}$,
then we have the following solution 
\begin{equation*}
x_{n}=x_{0}^{3^{n}}\prod\limits_{k=0}^{n-1}\left( a+b\right)
^{3^{n-k-1}}=x_{0}^{3^{n}}\left( a+b\right) ^{\left(
\sum\limits_{k=0}^{n-1}3^{n-k-1}\right) }=x_{0}^{3^{n}}\left( a+b\right) ^{%
\frac{3^{n}-1}{3-1}}
\end{equation*}%
which is solution of (\ref{const}). We here adopt the following conventions%
\begin{equation*}
\sum\limits_{k=m}^{m-l}s_{k}=0\text{ and }\prod\limits_{k=m}^{m-l}s_{k}=1,\
l\in \mathbb{N}.
\end{equation*}%
Hence, the above formulas are also valid for $n=0$.\newline
In this study, we use the general solutions of the two-dimensional system of
first-order linear difference equations with constant coefficients is given
by 
\begin{equation}
\left[ 
\begin{array}{c}
u_{n+1} \\ 
v_{n+1}%
\end{array}%
\right] =\left[ 
\begin{array}{cc}
a & b \\ 
c & d%
\end{array}%
\right] \left[ 
\begin{array}{c}
u_{n} \\ 
v_{n}%
\end{array}%
\right]  \label{veksys}
\end{equation}%
for all $n\in \mathbb{N}_{0}$. Let%
\begin{equation}
A=\left[ 
\begin{array}{cc}
a & b \\ 
c & d%
\end{array}%
\right] \text{.}  \label{A}
\end{equation}%
Then, to obtain the general solution of system (\ref{veksys}), we need to
find the $n$th power of the square matrix $A$. Therefore, we can write the
general solution of the system as the following 
\begin{equation}
\left[ 
\begin{array}{c}
u_{n} \\ 
v_{n}%
\end{array}%
\right] =\left[ 
\begin{array}{cc}
a & b \\ 
c & d%
\end{array}%
\right] ^{n}\left[ 
\begin{array}{c}
u_{0} \\ 
v_{0}%
\end{array}%
\right] .  \label{veksys1}
\end{equation}%
For some fundamental standard results on system (\ref{veksys}), see \cite%
{Kulenovic}. The following lemmas serve to solve system (\ref{veksys}).

\begin{lemma}
\label{lem01} Consider system (\ref{veksys}). If $ad-bc=0$, then the general
solution of system (\ref{veksys}) is given by 
\begin{equation}
\left[ 
\begin{array}{c}
u_{n} \\ 
v_{n}%
\end{array}%
\right] =\left( a+d\right) ^{n-1}\left[ 
\begin{array}{c}
au_{0}+bv_{0} \\ 
cu_{0}+dv_{0}%
\end{array}%
\right] .  \label{ad-bcsol}
\end{equation}
\end{lemma}

\begin{proof}
Suppose that $ad-bc=0$. Then, we first observe that $A^{2}=\left( a+d\right)
A$. By using this identity, we have 
\begin{equation*}
A^{n}=A^{2}A^{n-2}=\left( a+d\right) AA^{n-2}=\left( a+d\right) A^{n-1}
\end{equation*}%
for $n\geq 2$. Since the last equality is a first-order difference equation,
we can write 
\begin{eqnarray*}
A^{n} &=&\left( a+d\right) A^{n-1} \\
\left( a+d\right) A^{n-1} &=&\left( a+d\right) ^{2}A^{n-2} \\
\left( a+d\right) ^{2}A^{n-2} &=&\left( a+d\right) ^{3}A^{n-3} \\
&&\vdots \\
\left( a+d\right) ^{n-2}A^{2} &=&\left( a+d\right) ^{n-1}A.
\end{eqnarray*}%
By summing the above equalities, we find the $n$th power of $A$ as follows 
\begin{equation}
A^{n}=\left( a+d\right) ^{n-1}\left[ 
\begin{array}{cc}
a & b \\ 
c & d%
\end{array}%
\right] .  \label{ad-bc=0}
\end{equation}%
Finally, by using (\ref{ad-bc=0}) in (\ref{veksys1}), we have (\ref{ad-bcsol}%
).
\end{proof}

\begin{lemma}
\label{lem2} Consider system (\ref{veksys}). If $ad-bc\neq 0$, then the
following statements are true.\newline
(i) If $\left( a-d\right) ^{2}+4bc\neq 0$, then 
\begin{equation}
\left[ 
\begin{array}{c}
u_{n} \\ 
v_{n}%
\end{array}%
\right] =\left[ 
\begin{array}{c}
\frac{b}{\lambda _{1}-\lambda _{2}}\left( \frac{c}{\lambda _{1}-a}%
u_{0}+v_{0}\right) \lambda _{1}^{n}-\frac{b}{\lambda _{1}-\lambda _{2}}%
\left( \frac{c}{\lambda _{2}-a}u_{0}+v_{0}\right) \lambda _{2}^{n} \\ 
\frac{1}{\lambda _{1}-\lambda _{2}}\left( cu_{0}+\left( \lambda
_{1}-a\right) v_{0}\right) \lambda _{1}^{n}-\frac{1}{\lambda _{1}-\lambda
_{2}}\left( cu_{0}+\left( \lambda _{2}-a\right) v_{0}\right) \lambda _{2}^{n}%
\end{array}%
\right] ,  \label{unvng}
\end{equation}%
where $\lambda _{1}$ and $\lambda _{2}$ are eigenvalues of $A$.\newline
(ii) If $\left( a-d\right) ^{2}+4bc=0$, then 
\begin{equation}
\left[ 
\begin{array}{c}
u_{n} \\ 
v_{n}%
\end{array}%
\right] =\left( \frac{a+d}{2}\right) ^{n-1}\left[ 
\begin{array}{c}
\frac{a+d}{2}u_{0}+\left( \frac{a-d}{2}u_{0}+bv_{0}\right) n \\ 
\frac{a+d}{2}v_{0}+\left( cu_{0}+\frac{d-a}{2}v_{0}\right) n%
\end{array}%
\right] .  \label{veksyssol}
\end{equation}
\end{lemma}

\begin{proof}
Consider the square matrix $A$ given in (\ref{A}). Suppose that $ad-bc\neq 0$%
. Then, to calculate $A^{n}$, one can use any method. For example, the
discrete analogue of the Putzer algorithm (See \cite{Elaydi} p. 118). So, we
give $A^{n}$ by the following formulas 
\begin{equation}
A^{n}=\left\{ 
\begin{array}{cc}
\left[ 
\begin{array}{cc}
\frac{\left( a-\lambda _{2}\right) \lambda _{1}^{n}-\left( a-\lambda
_{1}\right) \lambda _{2}^{n}}{\lambda _{1}-\lambda _{2}} & b\frac{\lambda
_{1}^{n}-\lambda _{2}^{n}}{\lambda _{1}-\lambda _{2}} \\ 
c\frac{\lambda _{1}^{n}-\lambda _{2}^{n}}{\lambda _{1}-\lambda _{2}} & \frac{%
\left( d-\lambda _{2}\right) \lambda _{1}^{n}-\left( d-\lambda _{1}\right)
\lambda _{2}^{n}}{\lambda _{1}-\lambda _{2}}%
\end{array}%
\right] , & \text{if }\left( a-d\right) ^{2}+4bc\neq 0, \\ 
\left( \frac{a+d}{2}\right) ^{n-1}\left[ 
\begin{array}{cc}
\frac{a+d}{2}+n\frac{a-d}{2} & bn \\ 
cn & \frac{a+d}{2}+n\frac{d-a}{2}%
\end{array}%
\right] , & \text{if }\left( a-d\right) ^{2}+4bc=0,%
\end{array}%
\right.  \label{An}
\end{equation}%
\newline
where $\lambda _{1}$ and $\lambda _{2}$ are eigenvalues of $A$. By using (%
\ref{An}) in (\ref{veksys1}), we have (\ref{unvng}) and (\ref{veksyssol}).
\end{proof}

\begin{remark}
If $\left( a-d\right) ^{2}+4bc\neq 0$ and $a+d=0$, then there exists the
relation%
\begin{equation}
\lambda _{1}=\sqrt{a^{2}+bc}=-\lambda _{2}.  \label{L=-L}
\end{equation}%
By using (\ref{L=-L}) in (\ref{An}), we have 
\begin{equation}
A^{2n}=\left[ 
\begin{array}{cc}
\left( \sqrt{a^{2}+bc}\right) ^{2n} & 0 \\ 
0 & \left( \sqrt{a^{2}+bc}\right) ^{2n}%
\end{array}%
\right]  \label{A2n}
\end{equation}%
and 
\begin{equation}
A^{2n+1}=\left[ 
\begin{array}{cc}
a\left( \sqrt{a^{2}+bc}\right) ^{2n} & b\left( \sqrt{a^{2}+bc}\right) ^{2n}
\\ 
c\left( \sqrt{a^{2}+bc}\right) ^{2n} & d\left( \sqrt{a^{2}+bc}\right) ^{2n}%
\end{array}%
\right] .  \label{A2n+1}
\end{equation}%
Hence, by using (\ref{A2n})-(\ref{A2n+1}) in (\ref{veksys1}), we have the
general solution of system (\ref{veksys}) as follows 
\begin{equation*}
\left[ 
\begin{array}{c}
u_{2n} \\ 
v_{2n}%
\end{array}%
\right] =\left[ 
\begin{array}{c}
\left( \sqrt{a^{2}+bc}\right) ^{2n}u_{0} \\ 
\left( \sqrt{a^{2}+bc}\right) ^{2n}v_{0}%
\end{array}%
\right]
\end{equation*}%
and 
\begin{equation*}
\left[ 
\begin{array}{c}
u_{2n+1} \\ 
v_{2n+1}%
\end{array}%
\right] =\left[ 
\begin{array}{c}
\left( \sqrt{a^{2}+bc}\right) ^{2n}au_{0}+\left( \sqrt{a^{2}+bc}\right)
^{2n}bv_{0} \\ 
\left( \sqrt{a^{2}+bc}\right) ^{2n}cu_{0}+\left( \sqrt{a^{2}+bc}\right)
^{2n}dv_{0}%
\end{array}%
\right] .
\end{equation*}
\end{remark}

\section{Solvability and nontrivial solutions of system (\protect\ref{sys})}

This section contains our main results. We apply a quite practical method to
find the general solution in closed form of system (\ref{sys}).\newline
In the section introduction, the parameters and the initial values of system
(\ref{sys}) have been defined as arbitrary real numbers. These initial
values cause some trivial cases. It is easily seen from (\ref{sys}) that if $%
a=b=0$, then we have the trivial case $x_{n+1}=0$ , and so$\ y_{n+2}=0$ for
all real initial values and every $n\in \mathbb{N}_{0}$. Similarly, if $%
c=d=0 $, then we have the trivial case $y_{n+1}=0$ , and so$\ x_{n+2}=0$ for
all real initial values and every $n\in \mathbb{N}_{0}$.To avoid these
trivial cases, we suppose from now on that $\left( \left\vert a\right\vert
+\left\vert b\right\vert \right) \left( \left\vert c\right\vert +\left\vert
d\right\vert \right) \neq 0$. If $x_{0}=0$ or $y_{0}=0$, then we obtain the
trivial solution $x_{n}=y_{n}=0$ for every $n\in \mathbb{N}_{0}$. But,
choosing non-zero initial values does not guarantee a non-trivial solution.
That is, apart from these cases, there are still initial values that
produces an eventually trivial solution. The following lemma enable us to
determine such solutions of (\ref{sys}).

\begin{lemma}
Let $\left( \left\vert a\right\vert +\left\vert b\right\vert \right) \left(
\left\vert c\right\vert +\left\vert d\right\vert \right) \neq 0$ in system (%
\ref{sys}). Then the following statements are true.\newline
(i) If $ad-bc=0$, then the zero set of (\ref{sys}) is 
\begin{equation*}
Z_{0}=\left\{ \left( x_{0},y_{0}\right) :%
\begin{array}{c}
x_{0}y_{0}=0\text{ or }ax_{0}+by_{0}=0%
\end{array}%
\right\} .
\end{equation*}%
(ii) If $ad-bc\neq 0$ and $\left( a-d\right) ^{2}+4bc\neq 0$, then the zero
set of (\ref{sys}) is 
\begin{equation*}
Z_{1}=\left\{ \left( x_{0},y_{0}\right) :%
\begin{array}{c}
\left( \left( a-\lambda _{2}\right) \lambda _{1}^{n}-\left( a-\lambda
_{1}\right) \lambda _{2}^{n}\right) x_{0}+b\left( \lambda _{1}^{n}-\lambda
_{2}^{n}\right) y_{0}=0 \\ 
\text{ or }c\left( \lambda _{1}^{n}-\lambda _{2}^{n}\right) x_{0}+\left(
\left( d-\lambda _{2}\right) \lambda _{1}^{n}-\left( d-\lambda _{1}\right)
\lambda _{2}^{n}\right) y_{0}=0%
\end{array}%
,\ n\in \mathbb{N}_{0}\right\} .
\end{equation*}%
(iii) If $ad-bc\neq 0$ and $\left( a-d\right) ^{2}+4bc=0$, then the zero set
of (\ref{sys}) is 
\begin{equation*}
Z_{2}=\left\{ \left( x_{0},y_{0}\right) :%
\begin{array}{c}
\left( \frac{a+d}{2}+n\frac{a-d}{2}\right) x_{0}+bny_{0}=0 \\ 
\text{ or }cnx_{0}+\left( \frac{a+d}{2}+n\frac{d-a}{2}\right) y_{0}=0%
\end{array}%
,\ n\in \mathbb{N}_{0}\right\} .
\end{equation*}%
(iv) If $ad-bc\neq 0$, $\left( a-d\right) ^{2}+4bc\neq 0$ and $a+d=0$, then
the zero set of (\ref{sys}) is 
\begin{equation*}
Z_{3}=\left\{ \left( x_{0},y_{0}\right) :x_{0}y_{0}=0\text{ or }\left(
ax_{0}+by_{0}\right) =0\text{ or }\left( cx_{0}+dy_{0}\right) =0\right\} .
\end{equation*}
\end{lemma}

\begin{proof}
The cases when $x_{0}=0$ or $y_{0}=0$ were explained above. Let $%
x_{0}y_{0}\neq 0$. Then, we can assume that the first $n+1$ terms of both
the sequences $\left( x_{n}\right) $ and $\left( y_{n}\right) $ are nonzero.
That is, let us assume that 
\begin{equation}
\prod\limits_{k=0}^{n}x_{k}y_{k}\neq 0.  \label{prod}
\end{equation}%
Dividing both equations in (\ref{sys}) by (\ref{prod}), we obtain 
\begin{eqnarray}
\frac{x_{n+1}}{\prod\limits_{k=0}^{n}x_{k}y_{k}} &=&a\frac{x_{n}}{%
\prod\limits_{k=0}^{n-1}x_{k}y_{k}}+b\frac{y_{n}}{\prod%
\limits_{k=0}^{n-1}x_{k}y_{k}},  \label{psx} \\
\frac{y_{n+1}}{\prod\limits_{k=0}^{n}x_{k}y_{k}} &=&c\frac{x_{n}}{%
\prod\limits_{k=0}^{n-1}x_{k}y_{k}}+d\frac{y_{n}}{\prod%
\limits_{k=0}^{n-1}x_{k}y_{k}}  \label{psy}
\end{eqnarray}%
for all $n\in \mathbb{N}_{0}$. Note that the equations in (\ref{psx})-(\ref%
{psy}) constitute a linear system with respect to the variables $u_{n}$ and $%
v_{n}$ defined by 
\begin{equation}
u_{n}=\frac{x_{n}}{\prod\limits_{k=0}^{n-1}x_{k}y_{k}}\text{ and }v_{n}=%
\frac{y_{n}}{\prod\limits_{k=0}^{n-1}x_{k}y_{k}}.  \label{cv}
\end{equation}%
In fact, this linear system is system given by (\ref{veksys}). Hence, we can
exploit (\ref{veksys}) and (\ref{cv}) to solve (\ref{sys}). From (\ref{cv}),
we have 
\begin{eqnarray}
x_{n} &=&u_{n}\prod\limits_{k=0}^{n-1}x_{k}y_{k},  \label{cv1} \\
y_{n} &=&v_{n}\prod\limits_{k=0}^{n-1}x_{k}y_{k}.  \label{cv2}
\end{eqnarray}%
It is easy to see from (\ref{cv1}) and (\ref{cv2}) that we must choose $%
u_{n}=0$ or $v_{n}=0$ to obtain an eventually trivial solution of (\ref{sys}%
). That is, if $u_{n}=0$, then $x_{n}=0$. So, from (\ref{sys}), it follows
that $y_{n+1}=0$. Similarly, if $v_{n}=0$, then $y_{n}=0$. So, from (\ref%
{sys}), it follows that $x_{n+1}=0$.\newline
(i) It is clearly seen from (\ref{ad-bcsol}) that when $ad-bc=0$, if $%
u_{1}=au_{0}+bv_{0}=0$ (or equivalently $v_{1}=\frac{c}{a}\left(
au_{0}+bv_{0}\right) =0$), then $u_{n+1}=v_{n+1}=0$ for every $n\in \mathbb{N%
}_{0}$. \newline
(ii) If $ad-bc\neq 0$ and $\left( a-d\right) ^{2}+4bc\neq 0$, then, by using
(\ref{An}) in (\ref{veksys1}), we have 
\begin{equation}
\begin{array}{c}
u_{n}=\frac{\left( a-\lambda _{2}\right) \lambda _{1}^{n}-\left( a-\lambda
_{1}\right) \lambda _{2}^{n}}{\lambda _{1}-\lambda _{2}}u_{0}+b\frac{\lambda
_{1}^{n}-\lambda _{2}^{n}}{\lambda _{1}-\lambda _{2}}v_{0} \\ 
v_{n}=c\frac{\lambda _{1}^{n}-\lambda _{2}^{n}}{\lambda _{1}-\lambda _{2}}%
u_{0}+\frac{\left( d-\lambda _{2}\right) \lambda _{1}^{n}-\left( d-\lambda
_{1}\right) \lambda _{2}^{n}}{\lambda _{1}-\lambda _{2}}v_{0}%
\end{array}%
.  \label{LfL}
\end{equation}%
From (\ref{cv1}) and (\ref{LfL}), it follows that if 
\begin{equation}
\left( \left( a-\lambda _{2}\right) \lambda _{1}^{n}-\left( a-\lambda
_{1}\right) \lambda _{2}^{n}\right) u_{0}+b\left( \lambda _{1}^{n}-\lambda
_{2}^{n}\right) v_{0}=0,  \label{z1}
\end{equation}%
then $u_{n}=x_{n}=0$ for every $n\in \mathbb{N}_{0}$. By employing this
result in (\ref{sys}), we have $y_{n+1}=0$ for every $n\in \mathbb{N}_{0}$.
Similarly, from (\ref{cv2}) and (\ref{LfL}), it follows that if 
\begin{equation}
c\left( \lambda _{1}^{n}-\lambda _{2}^{n}\right) u_{0}+\left( \left(
d-\lambda _{2}\right) \lambda _{1}^{n}-\left( d-\lambda _{1}\right) \lambda
_{2}^{n}\right) v_{0}=0,  \label{z2}
\end{equation}%
then $v_{n}=y_{n}=0$ for every $n\in \mathbb{N}_{0}$. By employing this
result in (\ref{sys}) we have $x_{n+1}=0$ for every $n\in \mathbb{N}_{0}$.%
\newline
(iii) If $ad-bc\neq 0$ and $\left( a-d\right) ^{2}+4bc=0$, then we have 
\begin{equation}
\begin{array}{c}
u_{n}=\left( \frac{a+d}{2}\right) ^{n-1}\left[ \left( \frac{a+d}{2}+n\frac{%
a-d}{2}\right) u_{0}+bnv_{0}\right] , \\ 
v_{n}=\left( \frac{a+d}{2}\right) ^{n-1}\left[ cnu_{0}+\left( \frac{a+d}{2}+n%
\frac{d-a}{2}\right) v_{0}\right]%
\end{array}
\label{L=L}
\end{equation}%
by using (\ref{An}) in (\ref{veksys1}). From (\ref{cv1}) and (\ref{L=L}), it
follows that if 
\begin{equation}
\left( \frac{a+d}{2}+n\frac{a-d}{2}\right) u_{0}+bnv_{0}=0,  \label{z3}
\end{equation}%
then $u_{n}=x_{n}=0$ for every $n\in \mathbb{N}_{0}$. By employing this
result in (\ref{sys}), we have $y_{n+1}=0$ for every $n\in \mathbb{N}_{0}$.
Similarly, from (\ref{cv1}) and (\ref{L=L}), it follows that if 
\begin{equation}
cnu_{0}+\left( \frac{a+d}{2}+n\frac{d-a}{2}\right) v_{0}=0,  \label{z4}
\end{equation}%
then $v_{n}=y_{n}=0$ for every $n\in \mathbb{N}_{0}$. By employing this
result in (\ref{cv1}), we have $x_{n+1}=0$ for every $n\in \mathbb{N}_{0}$.%
\newline
(iv) If $a+d=0$, then the relation in (\ref{L=-L}) is satified. Using this
relation in (\ref{z1}) and (\ref{z2}), one can obtain desired result. 
\newline
Also, from (\ref{cv1}) and (\ref{cv2}), it follows for $n=0$ that 
\begin{equation}
u_{0}=x_{0}\text{ and }v_{0}=y_{0}.  \label{u=x}
\end{equation}%
This means that we can put $x_{0}$ in place of $u_{0}$ and $y_{0}$ in place
of $v_{0}$ in (\ref{z1}), (\ref{z2}), (\ref{z3}) and (\ref{z4}). This
completes the proof.
\end{proof}

We are now ready to find the general solution of system (\ref{sys}). Here,
we will use the formulas of $u_{n}$ and $v_{n}$ with $x_{0}$ and $y_{0}$
instead of $u_{0}$ and $v_{0}$ due to (\ref{u=x}). There are two cases of
the matrix $A$ to be considered as the following.\newline
The case $ad-bc=0$: In this case the first row in the matrix $A$ is linearly
dependent on the second one. Without loss of generality we may assume that
there exists the relation 
\begin{equation}
\left( c,d\right) =\frac{c}{a}\left( a,b\right) =\frac{d}{b}\left(
a,b\right) .  \label{rr}
\end{equation}%
By using (\ref{rr}) in system (\ref{veksys}), we have the relation%
\begin{equation}
v_{n+1}=\frac{c}{a}\left( au_{n}+bv_{n}\right) =\frac{c}{a}u_{n+1}
\label{rel}
\end{equation}%
for all $n\in \mathbb{N}_{0}$. Also, from the changes of variables in (\ref%
{cv}), we have the relation%
\begin{equation}
\frac{v_{n}}{u_{n}}=\frac{y_{n}}{x_{n}}  \label{vn/un}
\end{equation}%
for all $n\in \mathbb{N}_{0}$. (\ref{rel}) and (\ref{vn/un}) imply that 
\begin{equation}
\frac{y_{n+1}}{x_{n+1}}=\frac{c}{a}.  \label{y/x}
\end{equation}%
By employing (\ref{y/x}) in the first equation of (\ref{sys}), we have the
following equation 
\begin{equation}
x_{n+1}=\left( c+b\left( \frac{c}{a}\right) ^{2}\right) x_{n}^{3},
\label{xuv}
\end{equation}%
for all $n\in \mathbb{N}$. Hence, from (\ref{xuv}), we obtain%
\begin{equation}
x_{n}=x_{1}^{3^{n-1}}\left( c+b\left( \frac{c}{a}\right) ^{2}\right) ^{\frac{%
3^{n-1}-1}{2}}=\left( ax_{0}^{2}y_{0}+bx_{0}y_{0}^{2}\right)
^{3^{n-1}}\left( c+b\left( \frac{c}{a}\right) ^{2}\right) ^{\frac{3^{n-1}-1}{%
2}}.  \label{h1}
\end{equation}%
Also, by using (\ref{y/x}) and (\ref{h1}), we have 
\begin{equation}
y_{n}=\frac{c}{a}\left( ax_{0}^{2}y_{0}+bx_{0}y_{0}^{2}\right)
^{3^{n-1}}\left( c+b\left( \frac{c}{a}\right) ^{2}\right) ^{\frac{3^{n-1}-1}{%
2}}  \label{h2}
\end{equation}%
for $n\in \mathbb{N}$. Consequently, in the case $ad-bc=0$, the general
solution of (\ref{sys}) is given by the formulas in (\ref{h1})-(\ref{h2})
for all $n\in \mathbb{N}$.\newline
The case $ad-bc\neq 0$: In this case rows in the matrix $A$ are linearly
independent of each other. By employing (\ref{vn/un}) in the first equation
of (\ref{sys}) we have 
\begin{equation}
x_{n+1}=\left( a\frac{v_{n}}{u_{n}}+b\left( \frac{v_{n}}{u_{n}}\right)
^{2}\right) x_{n}^{3}  \label{xn+1}
\end{equation}%
for all $n\in \mathbb{N}_{0}$. Therefore, from (\ref{xn+1}), we have 
\begin{equation}
x_{n}=x_{0}^{3^{n}}\prod\limits_{k=0}^{n-1}\left( a\frac{v_{k}}{u_{k}}%
+b\left( \frac{v_{k}}{u_{k}}\right) ^{2}\right) ^{3^{n-k-1}}  \label{x1}
\end{equation}%
for all $n\in \mathbb{N}$. \newline
(i) If $\left( a-d\right) ^{2}+4bc\neq 0$, then, from (\ref{unvng}), we have
the ratio%
\begin{equation}
\frac{v_{n}}{u_{n}}=\frac{C_{3}\lambda _{1}^{n}-C_{4}\lambda _{2}^{n}}{%
C_{1}\lambda _{1}^{n}-C_{2}\lambda _{2}^{n}},  \label{v/ug}
\end{equation}%
where%
\begin{eqnarray}
C_{1} &=&\frac{b}{\lambda _{1}-\lambda _{2}}\left( \frac{c}{\lambda _{1}-a}%
x_{0}+y_{0}\right) ,  \label{c1} \\
C_{2} &=&\frac{b}{\lambda _{1}-\lambda _{2}}\left( \frac{c}{\lambda _{2}-a}%
x_{0}+y_{0}\right) ,  \label{c2} \\
C_{3} &=&\frac{1}{\lambda _{1}-\lambda _{2}}\left( cx_{0}+\left( \lambda
_{1}-a\right) y_{0}\right) ,  \label{c3} \\
C_{4} &=&\frac{1}{\lambda _{1}-\lambda _{2}}\left( cx_{0}+\left( \lambda
_{2}-a\right) y_{0}\right) ,  \label{c4}
\end{eqnarray}%
for all $n\in \mathbb{N}_{0}$. By employing (\ref{v/ug}) in (\ref{x1}), we
have the following formula%
\begin{equation}
x_{n}=x_{0}^{3^{n}}\prod\limits_{k=0}^{n-1}\left( a\frac{C_{3}\lambda
_{1}^{k}-C_{4}\lambda _{2}^{k}}{C_{1}\lambda _{1}^{k}-C_{2}\lambda _{2}^{k}}%
+b\left( \frac{C_{3}\lambda _{1}^{k}-C_{4}\lambda _{2}^{k}}{C_{1}\lambda
_{1}^{k}-C_{2}\lambda _{2}^{k}}\right) ^{2}\right) ^{3^{n-k-1}},  \label{xn}
\end{equation}%
for all $n\in \mathbb{N}_{0}$. By taking into account (\ref{vn/un}) and (\ref%
{xn}), we may write the formula for $y_{n}$ as the following 
\begin{equation}
y_{n}=\frac{C_{3}\lambda _{1}^{n}-C_{4}\lambda _{2}^{n}}{C_{1}\lambda
_{1}^{n}-C_{2}\lambda _{2}^{n}}x_{0}^{3^{n}}\prod\limits_{k=0}^{n-1}\left( a%
\frac{C_{3}\lambda _{1}^{k}-C_{4}\lambda _{2}^{k}}{C_{1}\lambda
_{1}^{k}-C_{2}\lambda _{2}^{k}}+b\left( \frac{C_{3}\lambda
_{1}^{k}-C_{4}\lambda _{2}^{k}}{C_{1}\lambda _{1}^{k}-C_{2}\lambda _{2}^{k}}%
\right) ^{2}\right) ^{3^{n-k-1}},  \label{yn}
\end{equation}%
where $C_{1}$, $C_{2}$, $C_{3}$ and $C_{4}$ are given in (\ref{c1})-(\ref{c4}%
), for all $n\in \mathbb{N}_{0}$. \newline
(ii) If $\left( a-d\right) ^{2}+4bc=0$, then, from (\ref{veksyssol}), we
have the ratio 
\begin{equation}
\frac{v_{n}}{u_{n}}=\frac{c_{3}+c_{4}n}{c_{1}+c_{2}n},  \label{v/u}
\end{equation}%
where 
\begin{eqnarray}
c_{1} &=&\frac{a+d}{2}x_{0},  \label{C1} \\
c_{2} &=&\frac{a-d}{2}x_{0}+by_{0},  \label{C2} \\
c_{3} &=&\frac{a+d}{2}y_{0},  \label{C3} \\
c_{4} &=&cx_{0}+\frac{d-a}{2}y_{0}.  \label{C4}
\end{eqnarray}%
By employing (\ref{v/u}) in (\ref{x1}) we have 
\begin{equation}
x_{n}=x_{0}^{3^{n}}\prod\limits_{k=0}^{n-1}\left( a\frac{c_{3}+c_{4}k}{%
c_{1}+c_{2}k}+b\left( \frac{c_{3}+c_{4}k}{c_{1}+c_{2}k}\right) ^{2}\right)
^{3^{n-k-1}},  \label{xnn}
\end{equation}%
where $c_{1}$, $c_{2}$, $c_{3}$ and $c_{4}$ are given in (\ref{C1})-(\ref{C4}%
), for all $n\in \mathbb{N}_{0}$. By taking into account (\ref{vn/un}) and (%
\ref{xnn}), we may write the formula 
\begin{equation}
y_{n}=\frac{c_{3}+c_{4}n}{c_{1}+c_{2}n}x_{0}^{3^{n}}\prod\limits_{k=0}^{n-1}%
\left( a\frac{c_{3}+c_{4}k}{c_{1}+c_{2}k}+b\left( \frac{c_{3}+c_{4}k}{%
c_{1}+c_{2}k}\right) ^{2}\right) ^{3^{n-k-1}},  \label{ynn}
\end{equation}%
where $c_{1}$, $c_{2}$, $c_{3}$ and $c_{4}$ are given in (\ref{C1}) and (\ref%
{C4}), for all $n\in \mathbb{N}_{0}$. \newline
(iii) If $\left( a-d\right) ^{2}+4bc\neq 0$ and $a+d=0$, then $\lambda _{1}=%
\sqrt{a^{2}+bc}=-\lambda _{2}$. By using this equality in (\ref{v/ug}), we
have 
\begin{equation*}
\frac{v_{n}}{u_{n}}=\frac{C_{3}\left( \sqrt{a^{2}+bc}\right)
^{n}-C_{4}\left( -\sqrt{a^{2}+bc}\right) ^{n}}{C_{1}\left( \sqrt{a^{2}+bc}%
\right) ^{n}-C_{2}\left( -\sqrt{a^{2}+bc}\right) ^{n}}=\frac{%
C_{3}-C_{4}\left( -1\right) ^{n}}{C_{1}-C_{2}\left( -1\right) ^{n}}
\end{equation*}%
and so%
\begin{eqnarray}
\frac{v_{2n}}{u_{2n}} &=&\frac{C_{3}-C_{4}}{C_{1}-C_{2}}=\frac{y_{0}}{x_{0}},
\label{y/x0} \\
\frac{v_{2n+1}}{u_{2n+1}} &=&\frac{C_{3}+C_{4}}{C_{1}+C_{2}}=\frac{%
cx_{0}+dy_{0}}{ax_{0}+by_{0}}.  \label{y/x1}
\end{eqnarray}%
On the other hand, from (\ref{xn+1}), we have%
\begin{equation}
x_{2n+1}=\left( a\frac{v_{2n}}{u_{2n}}+b\left( \frac{v_{2n}}{u_{2n}}\right)
^{2}\right) x_{2n}^{3},  \label{x2n+1d}
\end{equation}%
and 
\begin{equation*}
x_{2n+2}=\left( a\frac{v_{2n+1}}{u_{2n+1}}+b\left( \frac{v_{2n+1}}{u_{2n+1}}%
\right) ^{2}\right) x_{2n+1}^{3}
\end{equation*}%
which implies that 
\begin{equation}
x_{2n+2}=\left( a\frac{v_{2n+1}}{u_{2n+1}}+b\left( \frac{v_{2n+1}}{u_{2n+1}}%
\right) ^{2}\right) \left( a\frac{v_{2n}}{u_{2n}}+b\left( \frac{v_{2n}}{%
u_{2n}}\right) ^{2}\right) ^{3}x_{2n}^{3^{2}}  \label{x2n+2d}
\end{equation}%
for all $n\in \mathbb{N}_{0}$. Employing (\ref{y/x0})-(\ref{y/x1}) in (\ref%
{x2n+2d}), we obtain the following equation 
\begin{equation}
x_{2n+2}=\left( a\frac{cx_{0}+dy_{0}}{ax_{0}+by_{0}}+b\left( \frac{%
cx_{0}+dy_{0}}{ax_{0}+by_{0}}\right) ^{2}\right) \left( a\frac{y_{0}}{x_{0}}%
+b\left( \frac{y_{0}}{x_{0}}\right) ^{2}\right) ^{3}x_{2n}^{3^{2}}.
\label{x2n+2}
\end{equation}%
The solution of (\ref{x2n+2}) is given by 
\begin{equation}
x_{2n}=x_{0}^{3^{2n}}\left( \left( a\frac{cx_{0}+dy_{0}}{ax_{0}+by_{0}}%
+b\left( \frac{cx_{0}+dy_{0}}{ax_{0}+by_{0}}\right) ^{2}\right) \left( a%
\frac{y_{0}}{x_{0}}+b\left( \frac{y_{0}}{x_{0}}\right) ^{2}\right)
^{3}\right) ^{\frac{3^{2n}-1}{8}}  \label{adx2n}
\end{equation}%
for all $n\in \mathbb{N}_{0}$. Also, from (\ref{y/x0}), (\ref{x2n+1d}) and (%
\ref{adx2n}), we have 
\begin{eqnarray}
x_{2n+1} &=&x_{0}^{3^{2n+1}}\left( a\frac{y_{0}}{x_{0}}+b\left( \frac{y_{0}}{%
x_{0}}\right) ^{2}\right)  \notag \\
&\times &\left( \left( a\frac{cx_{0}+dy_{0}}{ax_{0}+by_{0}}+b\left( \frac{%
cx_{0}+dy_{0}}{ax_{0}+by_{0}}\right) ^{2}\right) \left( a\frac{y_{0}}{x_{0}}%
+b\left( \frac{y_{0}}{x_{0}}\right) ^{2}\right) ^{3}\right) ^{\frac{%
3^{2n+1}-3}{8}}  \label{adx2n+1}
\end{eqnarray}%
for all $n\in \mathbb{N}_{0}$. On the other hand, since 
\begin{eqnarray*}
y_{2n} &=&\frac{y_{0}}{x_{0}}x_{2n}, \\
y_{2n+1} &=&\frac{cx_{0}+dy_{0}}{ax_{0}+by_{0}}x_{2n+1}
\end{eqnarray*}%
we have 
\begin{equation}
y_{2n}=\frac{y_{0}}{x_{0}}x_{0}^{3^{2n}}\left( \left( a\frac{cx_{0}+dy_{0}}{%
ax_{0}+by_{0}}+b\left( \frac{cx_{0}+dy_{0}}{ax_{0}+by_{0}}\right)
^{2}\right) \left( a\frac{y_{0}}{x_{0}}+b\left( \frac{y_{0}}{x_{0}}\right)
^{2}\right) ^{3}\right) ^{\frac{3^{2n}-1}{8}}  \label{ady2n}
\end{equation}%
and%
\begin{eqnarray}
y_{2n+1} &=&\frac{cx_{0}+dy_{0}}{ax_{0}+by_{0}}x_{0}^{3^{2n+1}}\left( a\frac{%
y_{0}}{x_{0}}+b\left( \frac{y_{0}}{x_{0}}\right) ^{2}\right)  \notag \\
&\times &\left( \left( a\frac{cx_{0}+dy_{0}}{ax_{0}+by_{0}}+b\left( \frac{%
cx_{0}+dy_{0}}{ax_{0}+by_{0}}\right) ^{2}\right) \left( a\frac{y_{0}}{x_{0}}%
+b\left( \frac{y_{0}}{x_{0}}\right) ^{2}\right) ^{3}\right) ^{\frac{%
3^{2n+1}-3}{8}}.  \label{ady2n+1}
\end{eqnarray}%
From the above consideration, we obtain the following theorem which is the
main result of this section.

\begin{theorem}
Consider the system of difference equations (\ref{sys}), where the
parameters $a$, $b$, $c$, $d$ are real numbers such that $\left( \left\vert
a\right\vert +\left\vert b\right\vert \right) \left( \left\vert c\right\vert
+\left\vert d\right\vert \right) \neq 0$. Then, the following statements
hold.

\begin{enumerate}
\item If $ad-bc=0$ and $\left( x_{0},y_{0}\right) \notin Z_{0}$, then the
general solution to the system is given by formulas (\ref{h1}) and (\ref{h2}%
).

\item If $ad-bc\neq 0$, $\left( a-d\right) ^{2}+4bc=0$, $a+d\neq 0$ and $%
\left( x_{0},y_{0}\right) \notin Z_{1}$, then the general solution to the
system is given by formulas (\ref{xnn}) and (\ref{ynn}).

\item If $ad-bc\neq 0$, $\left( a-d\right) ^{2}+4bc\neq 0$, $a+d\neq 0$ and $%
\left( x_{0},y_{0}\right) \notin Z_{2}$, then the general solution to the
system is given by formulas (\ref{xn}) and (\ref{yn}).

\item If $ad-bc\neq 0$ , $\left( a-d\right) ^{2}+4bc\neq 0$, $a+d=0$ and $%
\left( x_{0},y_{0}\right) \notin Z_{3}$, then the general solution to the
system is given by formulas (\ref{adx2n}), (\ref{adx2n+1}), (\ref{ady2n})
and (\ref{ady2n+1}).
\end{enumerate}
\end{theorem}



\end{document}